\documentclass[10pt,twoside]{siamltex}
\usepackage{setspace}
\usepackage{color}
\usepackage{amsfonts,epsfig}
\usepackage{xy}
\input{xy}
\xyoption{all}
\newtheorem{remark}[theorem]{Remark}

\newtheorem{example}[theorem]{Example}

\begin{document}



\bibliographystyle{plain}
\title{
A note on a generalization of eigenvector centrality for bipartite
graphs and applications }

\author{
Peteris Daugulis\thanks{Department of Mathematics, Daugavpils
University, Daugavpils, Parades 1, Latvia
(peteris.daugulis@du.lv).}
}

\pagestyle{myheadings} \markboth{P.Daugulis}{A note on
generalization of eigenvector centrality for bipartite graphs and
applications} \maketitle

\begin{abstract}
Eigenvector centrality is a linear algebra based graph invariant
used in various rating systems such as webpage ratings for search
engines. A generalization of the eigenvector centrality invariant
is defined which is motivated by the need to design rating systems
for bipartite graph models of time-sensitive and other processes.
The linear algebra connection and some applications are described.
\end{abstract}


\section{Introduction} \label{intro-sec}

Eigenanalysis of graph adjacency matrices is \color{black} used in
network analysis and various heuristic rating systems such as
Google PageRank \cite{1} and Eigenfactor \cite{2}.

\begin{definition}
{\rm Let $\Gamma$ be an oriented weighted graph with the vertex
set $\{1,...,n\}$ and adjacency matrix $\textbf{A}=[a_{ij}]$
($a_{ij}\neq 0$ meaning there is an oriented edge of weight
$a_{ij}$ from $j$ to $i$).  We say that
$\textbf{c}=\left[c_{1}|...|c_{n}\right]^{T}\in
(\mathbb{R}^{+})^{n}$ is a \sl normed eigenvector centrality
sequence (NECS)\rm\ of $\Gamma$ provided $\textbf{c}$ is an
Euclidean-normed eigenvector for $\textbf{A}$ with a positive
eigenvalue.
 }
\end{definition}

The usefulness of the NECS invariant can be explained by
interpreting the eigenvector condition

$$
\label{11} \textbf{c}=\lambda \textbf{Ac},\
\lambda>0,\color{black}
$$
as a system of linear equations

\begin{eqnarray*}
\label{12}
\left\{%
\begin{array}{ll}
    c_{1}=\lambda \sum\limits_{j=1}^{n}a_{1j}c_{j} \\
    ... \\
    c_{n}= \lambda \sum\limits_{j=1}^{n}a_{nj}c_{j}\\
\end{array}%
\right.
\end{eqnarray*}
where the "rating" $c_{i}$ of the vertex $i$ is proportional with
the coefficient $\lambda$ for all $i$ to the sum of the "ratings"
of all vertices "related" to $i$ - having directed edges to
 $i$.

We remind standart definitions from the matrix theory. A real
valued matrix is called positive or nonnegative provided all its
entries are positive or nonnegative. A real valued square matrix
is called irreducible provided it is
not permutation equivalent to a block triangular matrix $\left[%
\begin{array}{c|c}
  \textbf{A} & \textbf{B} \\
  \hline
  \textbf{O} & \textbf{C} \\
\end{array}%
\right]$ where $\textbf{A}$ and $\textbf{C}$ are square matrices
and $\textbf{O}$ is a zero matrix.

\color{black}
 If $\textbf{A}$ is a positive or nonnegative irreducible matrix then by Perron or Perron-Frobenius theorem,
respectively, there exists a positive eigenvalue $\lambda_{1}$
(the dominant eigenvalue) \color{black} equal to the spectral
radius $\rho(\textbf{A})$ with a $1$-di\-men\-sio\-nal eigenspace
spanned by a positive eigenvector \cite{3}. Moreover all positive
eigenvectors of $\textbf{A}$ have $\lambda_1$ as the eigenvalue.
It follows that for such matrices $\textbf{A}$ the sequence
$\textbf{c}(\textbf{A})$ exists and is unique.

Since $\lambda_{1}=\rho(\textbf{A})$ it follows that
$\textbf{c}(\textbf{A})$ can be found by the power method starting
with a positive vector:
\begin{itemize}
\item take a positive vector such as
$\textbf{c}_{0}=[1|...|1]^{T}$,

\item make a loop $\textbf{c}_{r}=\textbf{A}\cdot
\Big(\frac{1}{||\textbf{c}_{r-1}||}\textbf{c}_{r-1}\Big)$,

\item compute
$\lim\limits_{r\rightarrow\infty}\textbf{c}_{r}=\textbf{c}(\textbf{A})$.
\end{itemize}
The convergence is geometric with rate
$|\frac{\lambda_{2}}{\lambda_{1}}|$ where $\lambda_{2}$ is the
second largest eigenvalue of $\textbf{A}$.

NECS is one of the several graph vertex centrality measures such
as degree, betweenness and closeness centrality used in the
analysis of networks \cite{4}. Note that the definition of NECS
assumes that the relation corresponding to the graph is defined in
a single set - the whole set of vertices.
\section{Positive weight relations}\label{pos}
In applications it is often necessary to model systems having objects of two kinds where
a weighted relation is defined on object pairs. This leads to
defining relations in ordered pairs of sets or considering
bipartite graphs. Given two finite sets $A,B$ we call a pair
$\pi=(R,w)$ where $R\subseteq A\times B$, $w:R\rightarrow
\mathbb{R}^{+}$ a \sl positive weight relation (PWR)\rm\ between
$A$ and $B$. Given a PWR $\pi$ between sets
$A=\{a_{1},...,a_{n}\}$ and $B=\{b_{1},...,b_{m}\}$ we
define its weight \color{black} matrix $\textbf{W}=[w_{ij}]_{m,n}$ where $$w_{ij}=\left\{%
\begin{array}{ll}
    w(a_{j},b_{i}),\ if\ (a_{j},b_{i})\in R \\
    0,\ if\ (a_{j},b_{i})\not\in R.  \\
\end{array}%
\right.    $$ $\textbf{W}$ defines a weighted oriented bipartite
graph $\Gamma(\textbf{W})$ with the vertex set $A\cup B$.

For $R\subseteq A\times B$ we remind the definition of
$R^{-1}\subseteq B\times A$: $(b,a)\in R^{-1}$ if and only if
$(a,b)\in R$.\color{black}

Given a PWR $\pi=(R,w)$ and a function $\varphi:
\mathbb{R}^{+}\rightarrow \mathbb{R}^{+}$ we define a reverse
\color{black} PWR $\pi'=(R^{-1},w')$ where $w':R^{-1}\rightarrow
\mathbb{R}^{+}$, $w'(b_{i},a_{j})=\varphi(w(a_{j},b_{i}))$. We
also get the corresponding reverse weight \color{black} matrix
$\textbf{W}'=[\varphi(w(a_{i},b_{j}))]_{n,m}$. The function
$\varphi$ \color{black} is chosen depending on the application.

\begin{example} Suppose $A$ is a set of workers, $B$ is a set of
tasks, $w_{ij}$ is time in which the worker $a_{j}$ performs the
task $b_{i}$.
\end{example}

\begin{example} Suppose $A$ is a set of consumers, $B$ is a set of
products, $w_{ij}$ is the amount of the product $b_{i}$ consumed
by $a_{j}$.
\end{example}

\section{A generalization of NECS} Given a pair of related \color{black} PWR
\begin{eqnarray*}
\left\{%
\begin{array}{ll}
    \pi=(R,w), \\
    \pi'=(R^{-1},w'),\\
\end{array}%
\right.
\end{eqnarray*}
we will generalize the NECS invariant. We call the pair
$(\textbf{a},\textbf{b})$, $\textbf{a}\in (\mathbb{R}^{+})^{n}$,
$\textbf{b}\in (\mathbb{R}^{+})^{m}$, the \sl normed eigenvector
bicentrality sequence (NEBS)\rm\ associated with the pair
$(\pi,\pi')$ provided the following system holds

\begin{equation}\label{initial}
\label{1} \left\{%
\begin{array}{ll}
    \textbf{b}=\lambda \textbf{W}\textbf{a}, \\
    \textbf{a}=\mu \textbf{W}'\textbf{b}, \\
    ||\textbf{a}||=||\textbf{b}||=1,
\end{array}%
\right.
\end{equation}
with $\lambda>0$, $\mu>0$. It follows that
$(\textbf{a},\textbf{b})$ satisfies the system \begin{equation}
\label{2} \left\{%
\begin{array}{ll}
    \textbf{b}=(\lambda\mu) (\textbf{W}\textbf{W}')\textbf{b}, \\
    \textbf{a}=(\lambda\mu) (\textbf{W}'\textbf{W})\textbf{a}, \\
||\textbf{a}||=||\textbf{b}||=1.
\end{array}%
\right.
\end{equation}

\begin{theorem} \label{main-th} Let $\pi=(R,w)$ and $\pi'=(R^{-1},w')$ be defined as in section \ref{pos}.
Let $\textbf{W}$ be a positive matrix. Then a \color{black} NEBS
$(\textbf{a},\textbf{b})$ associated to $(\pi,\pi')$ exists and is
unique.

\end{theorem}

\begin{proof} Since $\textbf{W}\textbf{W}'$ and $\textbf{W}'\textbf{W}$ are positive for each of the eigenvector \color{black} equations of \ref{2} there exists the
dominant positive eigenvalue and $1$-dimensional eigenspace with a
positive basis element according to the Perron-Frobenius theorem.
Moreover, the dominant eigenvalues are equal by the fact that the
sets of nonzero eigenvalues of $\textbf{W}\textbf{W}'$ and
$\textbf{W}'\textbf{W}$ are equal, see \cite{5}. If
$(\textbf{a},\textbf{b})$ is a solution of \ref{2} then
\begin{equation}
(\textbf{WW}')(\textbf{Wa})=\textbf{W}(\textbf{W}'\textbf{Wa})
=\frac{1}{\lambda\mu} (\textbf{W}\textbf{a}).
\end{equation}
It follows that $\textbf{Wa}=\alpha \textbf{b}$, $\alpha>0$, and,
by similar argument, $\textbf{W}'\textbf{b}=\beta \textbf{a}$,
$\beta>0$, $\alpha\beta\lambda\mu=1$. Thus a solution to
\ref{initial} exists and is unique.
\end{proof}

\begin{remark} It follows that the computation of NEBS can be done
using the power method and the convergence is geometric in
general.
\end{remark}

\begin{remark} The statement of the theorem \ref{main-th} is also
true by the same argument if a weaker condition for the matrices
$\textbf{W}$ and $\textbf{W}'$ holds: both matrices $\textbf{WW}'$
and $\textbf{W}'\textbf{W}$ are nonnegative and irreducible.
\end{remark}
\color{black}

\section{On the function $\varphi$} The role of $\varphi:\mathbb{R}^{+}\rightarrow
\mathbb{R}^{+}$ is to encode a suitable weight function for the
reverse relation $(R^{-1},w')$ given the initial PWR $(R,w)$.
$\varphi$ should be chosen using rigorous or heuristic
considerations related to (a) the models, (b) rating arguments and
(c) functional properties of $\varphi$ such as monotonicity,
concavity and values in specific intervals. Two models and
possible choices of $\varphi$ are given in the next section.

We now give two propositions showing the effect of $\varphi$ on
NEBS. The first propositions says that scalar multiples of a given
$\varphi$ produce the same NEBS.

\begin{theorem} \label{varphi_1} Let $\pi=(R,w)$ be a PWR with a positive weight matrix $\textbf{W}$, $R\subseteq A\times
B$. Let $\varphi_{i}:\mathbb{R}^{+}\rightarrow \mathbb{R}^{+}$ for
$i\in \{1,2\}$ be such that $\varphi_{1}=\gamma\varphi_{2}$ for
some $\gamma\in \mathbb{R}^{+}$. Define two PWR pairs
$(\pi,\pi'_{i})$ for $i\in \{1,2\}$ using $\varphi_{i}$. Then the
NEBS of $(\pi,\pi'_{1})$ and $(\pi,\pi'_{2})$ are equal.

\end{theorem}

\begin{proof} Let $\textbf{W}'_{i}$ be the
reverse weight matrix for the PWR pair $(\pi,\pi'_{i})$ for $i\in
\{1,2\}$. One can see that $\textbf{W}'_{1}=\gamma
\textbf{W}'_{2}$, $\textbf{WW}'_{1}=\gamma \textbf{WW}'_{2}$ and
$\textbf{W}'_{1}\textbf{W}=\gamma \textbf{W}'_{2}\textbf{W}$. It
follows that the normed positive eigenvectors for (a)
$\textbf{W}'_{1}\textbf{W}$ and $\textbf{W}'_{2}\textbf{W}$ and
(b) $\textbf{WW}'_{1}$ and $\textbf{W}\textbf{W}'_{2}$ are equal.
\end{proof}

The next proposition shows that given a weight matrix $\textbf{W}$
in a general position and an arbitrary positive normed vector
$\textbf{a}$ one can choose $\varphi$ so that $\textbf{a}$ is a
part of NEBS for the corresponding PWR pair $(\pi,\pi')$. Similar
statement can be proved for the other part of NEBS.

\begin{theorem} \label{varphi_2} Let $\textbf{W}$ be a positive $m\times n$ matrix having distinct entries.
Let $\pi=(R,w)$ be a PWR with $\textbf{W}$ as its weight matrix,
$R\subseteq A\times B$, $|A|=n$, $|B|=m$. Let $\textbf{a}\in
(\mathbb{R}^{+})^{n}$, $||\textbf{a}||=1$. Then there exists
 $\varphi:\mathbb{R}^{+}\rightarrow \mathbb{R}^{+}$ such
that the corresponding pair $(\pi,\pi')$ has $\textbf{a}$ as a
part of its NEBS $(\textbf{a},\textbf{b})$.

\end{theorem}

\begin{proof} We first show that it is possible to choose a positive matrix $\textbf{W}'$ so that
$\textbf{a}$ satisfies the matrix equation

\begin{equation}\label{3}
\textbf{a}=\textbf{W}'\textbf{W}\textbf{a}.
\end{equation}
 Let
$\textbf{W}=[w_{ij}]_{m,n}$, $\textbf{a}=[A_{1}|...|A_{n}]^{T}$
then $\textbf{Wa}=[\sum\limits_{l=1}^{n}w_{il}A_{l}]$. Let
$d_{i}=\sum\limits_{l=1}^{n}w_{il}A_{l}$,
$s=\sum\limits_{i=1}^{n}d_{i}$ and $t_{i}=A_{i}/s$. Let
$\textbf{W}'=[w'_{ij}]_{n,m}$ with $w'_{ij}=t_{i}$. One can check
that \ref{3} holds. Thus
$(\textbf{a},\frac{1}{||\textbf{Wa}||}\textbf{Wa})$ is the NEBS
for the pair $(\pi,\pi')$ constructed using $\textbf{W}$ and
$\textbf{W}'$, $\lambda=1/||\textbf{Wa}||$, $\mu=||\textbf{Wa}||$.
Define $\varphi$ by setting $w'_{ji}=\varphi(w_{ij})$ and
extending $\varphi$ arbitrarily for other values of the domain
$\mathbb{R}^{+}$. Then $\textbf{W}'$ is obtained using
$\textbf{W}$ and $\varphi$.
\end{proof}

\color{black}

\section{Applications}
\subsection{Problem solving} Let $A=\{a_{1},...,a_{m}\}$ be a set
of students and $B=\{b_{1},...,b_{n}\}$ be a set of problems.
Suppose each student has solved each problem and the solving time
has been measured. Denote by $w_{ij}$ the time in which $a_j$ has
solved $b_i$. One can pose the problem of comparing (rating)
relative difficulty of the problems and relative problem solving
speed of the students. A simple \color{black} way would be to
compare the time sums or averages $\bar{a}_{j}$ and $\bar{b}_{i}$
between objects of the same type (students or \color{black}
problems) defined as follows:
\begin{eqnarray*}
\bar{a}_{j}=\frac{1}{|B\color{black}|}\sum_{i}w_{ij},\\
\bar{b}_{i}=\frac{1}{|A\color{black}|}\sum_{j}w_{ij}.
\end{eqnarray*}
This may fail to distinguish objects as the next example shows.

\begin{example} \label{ex_1}Let $|A|=|B|=2$, $w_{11}=2$, $w_{21}=2$, $w_{12}=3$,
$w_{22}=1$. $\Gamma(\textbf{W})$ is shown in Fig.1.


\begin{eqnarray*}
\xymatrix{
a_{1}\ar[rrr]^<<<<<<<<<{2}\ar[drrr]^<<<<<<<{2}&&&b_{1}\\
a_{2}\ar[rrr]^<<<<<<<<<{1}\ar[urrr]^<<<<<<<<<{3} &&&b_{2}
 }
\end{eqnarray*}
\begin{center}
Fig.1. - $\Gamma(\textbf{W})$ for Example 5.1.
\end{center}
While we can rate the problems using time averages - $b_{1}$ has
the average solving time $\bar{b_{1}}=2.5$ and $b_{2}$ has
$\bar{b_{2}}=1.5<\bar{b_{1}}$, the students can not be rated since
they have equal time averages - $\bar{a_{1}}=\bar{a_{2}}=2$.
\end{example}

We need to refine the rating system. Let us try to define a rating
system for solvers and problems according to the following
assumptions:
\begin{itemize}
    \item the problem ratings are additive with respect to
    the solver ra\-tings and vice versa,
    \item the contribution of the $a_l$ to the rating of
    $b_i$ is proportional to the rating of $a_l$ and $w\color{black}_{il}$
    (since the harder problems will have larger solving times),
    \item the contribution the $b_k$ to the rating of
    $a_j$ is proportional to the rating of $b_k$ and
    $1/w_{kj}\color{black}$ (since the best solvers will have smaller solving times).
\end{itemize}
Define $2$ rating vectors
$\textbf{a}=[A_{1}|...|A_{m}]^{T}\color{black}$ and
$\textbf{b}=[B_{1}|...|B_{n}]^{T}\color{black}$ as follows:
\begin{eqnarray*}
\left\{%
\begin{array}{ll}
    B_{i}=\lambda\sum\limits_{l=1}^{m}w_{il}A_{l},\\
    A_{j}=\mu\sum\limits_{k=1}^{n}\Big(\frac{1}{w_{kj}}\Big)B_{k},\color{black} \\
    ||\textbf{a}||=||\textbf{b}||=1,\\
\end{array}%
\right. \iff
\left\{%
\begin{array}{ll}
    \textbf{b}=\lambda \textbf{W}\textbf{a},\\
    \textbf{a}=\mu \textbf{W}'\textbf{b},\\
    ||\textbf{a}||=||\textbf{b}||=1,\\
\end{array}%
\right.
\end{eqnarray*}
where
\begin{eqnarray*}
\left\{%
\begin{array}{ll}
    \textbf{W}=[w_{ij}]_{n,m},\\
    \textbf{W}'=[1/w_{ji}]_{m,n}.\\
\end{array}%
\right.
\end{eqnarray*}
Thus in this case $\varphi(x)=1/x$. We see that the rating
sequence $(\textbf{a},\textbf{b})$ is exactly NEBS. Since the
rating of a student should be a decreasing function of her/his
solving times we could consider choosing other decreasing
functions as $\varphi$.

\color{black}

\begin{example}In example \ref{ex_1} we have
$\textbf{W}=\left[%
\begin{array}{c|c}
  2 & 3 \\
  \hline
  2 & 1 \\
\end{array}%
\right]$, $\textbf{WW}'=\left[%
\begin{array}{c|c}
  2 & 4 \\
  \hline
  4/3 & 2 \\
\end{array}%
\right]$,

$\textbf{a}\approx[0.65|0.75]^{T}$,
$\textbf{b}\approx[0.87,0.5]^{T}$.
\end{example}

\subsection{Product and consumer ratings} Let $A=\{a_{1},...,a_{m}\}$ be a set
of consumers and $B=\{b_{1},...,b_{n}\}$ be a set of products.
Suppose each consumer has bought (consumed) all products. Denote
by $w_{ij}$ the amount of $b_i$ bought by $a_j$. One can pose the
problem of comparing (rating) relative desirability of the
products and relative consumption rate of the consumers. Let us
define a rating system for consumers and products according to the
following assumptions:
\begin{itemize}
    \item the product ratings are additive with respect to
    the consumer ra\-tings and vice versa,
    \item the contribution of $a_l$ to the rating of
    $b_i$ is proportional to the rating of $a_l$ and $w_{il}$
    (since the consumed amounts of better products will be larger),
    \item the contribution $b_k$ to the rating of
    $a_j$ is proportional to the rating of $b_k$ and
    $w_{kj}$ (since the best consumers will have larger amounts of consumed products).
\end{itemize}
\color{black} In this model we have $\varphi(x)=x$.  We can
compute the NEBS with $\textbf{W}=[w_{ij}]_{n,m}$ and
$\textbf{W}'=\textbf{W}^{T}$.

\begin{remark}
The NEBS still fails to provide different ratings in special
cases. This happens \color{black} if $\textbf{W}$ is positive and
any of the matrices $\textbf{W}\textbf{W}'$ or
$\textbf{W}'\textbf{W}$ has equal row sums, e.g. it is a Latin
square. To explain this suppose $\textbf{W}\textbf{W}'$ has equal
row sums. Then the constant vector $[1|...|1]^{T}$ is an
eigenvector for $\textbf{W}\textbf{W}'$ with a positive
eigenvalue. Since the positive vector $\textbf{b}$ satisfies the
eigenvector condition \ref{2}
$$\textbf{b}=(\lambda\mu)
(\textbf{W}\textbf{W}')\textbf{b}$$ and all positive eigenvectors
of $\textbf{WW}'$ belong to the dominant eigenvalue it follows
that $\textbf{b}$ is a multiple of  $[1|...|1]^{T}$ and thus all
elements of $B$ have equal ratings. \color{black}
\end{remark}

\section{Conclusion}
We have generalized the eigenvector centrality sequence for
bipartite graphs by defining a pair of related sequences. The
generalization is motivated by the need to define a rating system
for two sets of objects involved in a single process such as
problem solving which would be subtler than lists of average
incoming or outgoing weights. The new invariants are computed as
positive normed eigenvectors for the dominant eigenvalues of
certain matrices, they exist and are unique if the matrices
involved satisfy conditions of the Perron theorem. Further
research may be done (a) to explore and implement possible
practical applications, (b) to explore different choices of
$\varphi$ in these applications, (c) to find sufficient conditions
on $W$\color{black} and/or $\varphi$ for which NEBS exists and is
unique.

\newpage


\end{document}